\documentclass[11pt,reqno]{amsart}

\usepackage[utf8]{inputenc}
\usepackage[T1]{fontenc}
\usepackage[english]{babel}
\usepackage{amsmath, amssymb, amsthm, amsfonts}
\usepackage{mathrsfs}
\usepackage{mathtools}
\usepackage{geometry}
\usepackage{xcolor}
\usepackage{enumitem}
\usepackage{hyperref}

\newtheorem{theorem}{Theorem}[section]
\newtheorem{lemma}[theorem]{Lemma}

\theoremstyle{remark}
\newtheorem{remark}[theorem]{Remark}
\newtheorem{definition}[theorem]{Definition}

\title[]{Critical wave equation with nonlinear $E(t)u_t$ damping}
\author{M. M. Cavalcanti, V. N. Domingos Cavalcanti, J. C. O. Faria , C. A. Okawa}
\address{Department of Mathematics, State University of Maringá, Maringá, PR, Brazil.}
\email{mmcavalcanti@uem.br;~vndcavalcanti@uem.br; jcofaria@uem.br; cintyaokawa@gmail.com}

\begin{document}
\begin{abstract}
We investigate a semilinear wave equation with energy--critical
nonlinearity and a nonlinear damping mechanism driven by the total
energy of the system. The model combines the quintic defocusing term
with a time--dependent dissipation of the form $E(t)u_t$, which introduces
a nonstandard feedback structure coupling the dynamics and the energy
functional.

Weak solutions are constructed via Galerkin approximations, with the
passage to the limit relying on uniform energy estimates and compactness
arguments. Special attention is devoted to the critical nature of the
nonlinearity, where concentration phenomena prevent purely energy--based
methods from yielding refined spacetime control. This difficulty is
resolved by incorporating nonhomogeneous Strichartz estimates together
with smoothly truncated spectral approximations, ensuring uniform bounds
at the dispersive level.

Finally, we establish polynomial decay rates for the energy by adapting
Nakao’s method to the present nonlinear dissipative framework. The results
highlight the stabilizing effect of the energy--dependent damping and its
interaction with the critical wave dynamics.
\end{abstract}

\maketitle
\tableofcontents
\section{Introduction}

In 1989, Balakrishnan and Taylor \cite{Balakrishnan}, inspired by Zhang's work \cite{Zhang}, introduced several models related to the study of dissipative phenomena, referred to as nonlocal dissipations, in flight structures. The single mode behavior of such models can be described by the following ODE \cite[p. FDC-3]{Balakrishnan}:
\begin{equation}\label{eq:ODE_BT}
\ddot x(t) + \lambda x(t) + d D(x(t),\dot x(t))=0,\quad t>0,
\end{equation}
where $x(t)$ represents the displacement of a point $x$ in the flight structure at time $t$ and $\lambda,\, d$ are positive constants. This model has attracted the attention of researchers worldwide due to the slow decay produced by this kind of damping mechanism.

Now, we are going to discuss some distributed parameter models. Let $\Omega$ be a bounded domain in $\mathbb{R}^n$ with smooth boundary and let us consider, for instance, the linear wave equation subject to a nonlinear and nonlocal damping term of Balakrishnan-Taylor type:
\begin{equation}\label{eq1.1'}
\partial_{tt} u - \Delta u + E_u^w(t) \partial_t u = 0 \quad \hbox{ in } \Omega \times \mathbb{R}_+,
\end{equation}
with Dirichlet homogeneous boundary conditions and weak initial data $(u_0,u_1)\in H_0^1(\Omega) \times L^2(\Omega)$. The energy associated to weak solutions to problem \eqref{eq1.1'} is given by
$$E_u^w(t):= \frac12 \|\partial_t u(t)\|_2^2 + \frac12\|\nabla u(t)\|_2^2.$$
A simple computation shows that
\begin{equation*}
\frac{d}{dt} E^w_u(t) + E_u^w(t) \|\partial_t u(t)\|_{L^2(\Omega)}^2=0
\end{equation*}
for all $t\geq 0$. Since $\|\partial_t u(t)\|_{L^2(\Omega)}^2 \le 2E_u^w(t)$, we obtain
\begin{equation}\label{eq1.2}
\frac{d}{dt} E_u^w(t) + 2\left(E^w_u(t)\right)^2 \geq 0
\end{equation}
for any $t\geq 0.$ Thus, integrating inequality \eqref{eq1.2}, we infer that
\begin{equation}\label{eq1.3}
E^w_u(t) \geq \left( \frac{1}{E^w_u(0)} + 2t\right)^{-1}
\end{equation}
for any $t\geq 0,$ which means that the energy $E^w_u(t)$ cannot decay fast to zero. In fact, it can be proven that the energy associated with problem \eqref{eq1.1'} satisfies the following inequality:
\begin{equation*}
\left(4t+\frac{1}{E^w_u(0)}\right)^{-1}\leq E_u^w(t)\leq \left(\frac{1}{\mu}(t-1)^{+}+\frac{1}{E_u^w(0)}\right)^{-1}
\end{equation*}
for some $\mu>0$ and for any $t\geq 0$, which means that $1/t$ is the optimal decay rate estimate for this kind of prototype.

Motivated by these physical models and the inherently slow dissipation they produce, the primary goal of the present paper is to investigate the interaction between the nonlocal Balakrishnan-Taylor damping and an energy-critical nonlinearity. Specifically, we study the well-posedness and regularity of the defocusing critical wave equation in a three-dimensional bounded domain $\Omega$:
\begin{equation}\label{eq:main-intro}
\begin{cases}
&u_{tt}-\Delta u+u^5+E(t)u_t=0,  \hbox{in } (0,\infty)\times\Omega,~u=0, \\
& \hbox{on } (0,\infty)\times\partial\Omega,\\
&(u(0),u_t(0))=(u_0,u_1)\in H_0^1(\Omega)\times L^2(\Omega).
\end{cases}
\end{equation}

The analysis of energy-critical wave equations has been a cornerstone in the theory of dispersive PDEs. For the pure wave equation, global well-posedness and scattering were established in the foundational works of Grillakis \cite{Grillakis} and Shatah and Struwe \cite{ShatahStruwe}. These breakthroughs were built upon the deep understanding of defect measures by G'erard \cite{Gerard} and the profile decomposition techniques developed by Bahouri and G'erard \cite{BahouriGerard} and Kenig and Merle \cite{KenigMerle}. These tools are essential to rule out energy concentration and the formation of singularities. However, the study of critical wave equations on bounded domains introduces severe technical obstacles. As demonstrated by Fefferman's celebrated ball multiplier theorem \cite{Fefferman}, sharp spectral projectors (such as standard Galerkin truncations) are unbounded in $L^p(\Omega)$ spaces for $p \neq 2$. This prevents the direct extraction of uniform critical Strichartz bounds from finite-dimensional approximations. While linear Strichartz estimates on manifolds with boundaries are well-understood—thanks to the foundational works of Smith and Sogge \cite{SmithSogge1995} and the deep results by Blair, Smith, and Sogge \cite{BlairSmithSogge}, who established that dispersive bounds hold by delicately analyzing the $L^p$ norms of spectral clusters to overcome the geometric problem of wave caustics concentrating energy at the boundary  —their application to critical nonlinearities requires delicate regularization techniques. Furthermore, by means of the Christ-Kiselev lemma \cite{ChristKiselev}, these homogeneous estimates successfully extend to the nonhomogeneous linear wave equation, providing the necessary bounds to handle the source terms via Duhamel's principle.To overcome these difficulties, the core contribution of this paper relies on a smooth spectral approximation of Littlewood-Paley type, inspired by the framework of Burq, Lebeau, and Planchon \cite{BurqLebeauPlanchon}. By smoothing the spectral cut-off, we completely bypass the Fefferman multiplier limitation, obtaining uniform \textit{a priori} bounds in the critical Strichartz spaces $L^5_t L^{10}_x$ and $L^4_t L^{12}_x$. This approach ensures that the sequence of approximate solutions is immune to profile formation, allowing a rigorous passage to the limit to construct a unique global Shatah-Struwe solution. Furthermore, we demonstrate that this high-regularity setting allows us to derive uniform bounds for the higher-order energy, ensuring that no derivative loss occurs.Once the global strong well-posedness is firmly established, we address the asymptotic behavior of the system. We show that the algebraic decay rate $\mathcal{O}(t^{-1})$, inherent to the Balakrishnan-Taylor damping, is perfectly preserved in the presence of the critical defocusing nonlinearity by means of Nakao's difference inequality method \cite{Nakao}. To the best of our knowledge, this work introduces a novel framework to bypass the classical Fefferman spectral obstruction in the context of energy-critical damping.

The paper is organized as follows: In Section 2, we construct energy weak solutions using the standard sharp Galerkin method, which is perfectly suited for $L^2$-based energy identities. In Section 3, we introduce the smooth spectral multipliers and prove the uniform critical bounds leading to the global Shatah-Struwe solutions. Finally, Section 4 is devoted to the asymptotic behavior of the energy.
\section{Existence of energy weak solutions}

We prove the existence of a global energy weak solution to the defocusing energy-critical wave equation:
\begin{equation}\label{eq:main-weak}
\begin{cases}
u_{tt}-\Delta u+u^5+E(t)u_t=0, & (t,x)\in (0,T)\times\Omega,\\
u=0, & (t,x)\in (0,T)\times\partial\Omega,\\
(u(0),u_t(0))=(u_0,u_1)\in H_0^1(\Omega)\times L^2(\Omega),
\end{cases}
\end{equation}
where $\Omega \subset \mathbb{R}^3$ is a smooth bounded domain and the energy is defined by
\begin{equation}\label{eq:Edef}
E(t):=\frac12\Big(\|\nabla u(t)\|_{L^2}^2+\|u_t(t)\|_{L^2}^2\Big)
+\frac16\|u(t)\|_{L^6}^6.
\end{equation}

\subsection{Galerkin approximation}

Let $\{\varphi_k\}_{k\ge1}$ be the Dirichlet eigenfunctions of $-\Delta$ in $\Omega$
and $V_m:=\mathrm{span}\{\varphi_1,\dots,\varphi_m\}$.
Let $P_m$ be the $L^2$-orthogonal projector onto $V_m$.
We look for $u_m(t)\in V_m$ solving the approximate problem
\begin{equation}\label{eq:galerkin}
u_{m,tt}-\Delta u_m+P_m(u_m^5)+E_m(t)u_{m,t}=0,
\qquad
(u_m(0),u_{m,t}(0))=(P_m u_0,P_m u_1),
\end{equation}
where
\begin{equation}\label{eq:Emdef}
E_m(t):=\frac12\Big(\|\nabla u_m(t)\|_{L^2}^2+\|u_{m,t}(t)\|_{L^2}^2\Big)
+\frac16\|u_m(t)\|_{L^6}^6.
\end{equation}
Since \eqref{eq:galerkin} is an ODE in finite dimension with a locally Lipschitz nonlinearity, it admits a unique smooth
solution on $[0,T]$.

\subsection{Energy identity and uniform bounds}

\begin{lemma}[Energy identity at the Galerkin level]\label{lem:energy-gal}
For every $m$, the Galerkin energy satisfies
\begin{equation}\label{eq:Em-id}
\frac{d}{dt}E_m(t)+E_m(t)\,\|u_{m,t}(t)\|_{L^2}^2=0
\qquad\text{for a.e. }t\in(0,T).
\end{equation}
In particular, $E_m$ is nonincreasing and $0\le E_m(t)\le E_m(0)\le C\,\|(u_0,u_1)\|_{H_0^1\times L^2}^2$,
uniformly in $m$.
\end{lemma}

\begin{proof}
Take the $L^2$-inner product of \eqref{eq:galerkin} with $u_{m,t}$.
Using Dirichlet boundary conditions, we have
$\langle -\Delta u_m,u_{m,t}\rangle=\frac{d}{dt}\frac12\|\nabla u_m\|_2^2$,
and $\langle u_{m,tt},u_{m,t}\rangle=\frac{d}{dt}\frac12\|u_{m,t}\|_2^2$.
Moreover, since $u_{m,t}\in V_m$ and $P_m$ is orthogonal,
\[
\langle P_m(u_m^5),u_{m,t}\rangle=\langle u_m^5,u_{m,t}\rangle
=\frac{d}{dt}\frac16\|u_m\|_6^6.
\]
Finally, $\langle E_m(t)u_{m,t},u_{m,t}\rangle=E_m(t)\|u_{m,t}\|_2^2$.
Summing up gives \eqref{eq:Em-id}.
\end{proof}

As a consequence, for each fixed $T>0$ we have the uniform bounds
\begin{equation}\label{eq:uniform-energy}
\sup_{m}\Big(\|u_m\|_{L^\infty(0,T;H_0^1)}+\|u_{m,t}\|_{L^\infty(0,T;L^2)}\Big)\le C(E_0),
\end{equation}
where $E_0:=E(0)$ is computed from $(u_0,u_1)$.

\subsection{Helly/compactness for the coefficient $E_m(t)$}

\begin{lemma}[BV bound and Helly selection]\label{lem:helly}
The family $\{E_m\}$ is uniformly bounded in $BV(0,T)\cap L^\infty(0,T)$.
In particular, there exist a subsequence (not relabeled) and a function
$E\in BV(0,T)\cap L^\infty(0,T)$ such that
\begin{equation}\label{eq:Em-strong-L1}
E_m\to E\quad\text{strongly in }L^1(0,T)\ \text{and a.e. in }(0,T).
\end{equation}
Moreover, for every $1\le p<\infty$,
\begin{equation}\label{eq:Em-strong-Lp}
E_m\to E\quad\text{strongly in }L^p(0,T).
\end{equation}
\end{lemma}

\begin{proof}
From \eqref{eq:Em-id}, we have
\[
E_m'(t)=-E_m(t)\|u_{m,t}(t)\|_2^2\le0,
\]
hence each $E_m$ is monotone non-increasing. In particular,
\[
|E_m'(t)| = -E_m'(t)=E_m(t)\|u_{m,t}(t)\|_2^2.
\]

Using again \eqref{eq:Em-id}, we obtain
\[
\int_0^T |E_m'(t)|\,dt
= E_m(0)-E_m(T)
\le E_m(0)
\le C(E_0),
\]
uniformly in $m$. Therefore $\{E_m\}$ is uniformly bounded in $BV(0,T)$.

On the other hand, since $E_m(t)\ge0$ and $E_m(t)\le E_m(0)$,
the family is also uniformly bounded in $L^\infty(0,T)$:
\[
\|E_m\|_{L^\infty(0,T)} \le E_m(0)\le C(E_0).
\]

Hence $\{E_m\}$ is bounded in $BV(0,T)\cap L^\infty(0,T)$.
By Helly's selection theorem, there exist a subsequence (not relabeled)
and a function $E\in BV(0,T)$ such that
\[
E_m(t)\to E(t) \quad \text{for a.e. } t\in(0,T).
\]

Since the sequence is uniformly bounded in $L^\infty(0,T)$,
Lebesgue's dominated convergence theorem yields
\[
E_m\to E \quad \text{strongly in } L^1(0,T),
\]
which proves \eqref{eq:Em-strong-L1}.

Finally, combining the strong $L^1$ convergence with the uniform
$L^\infty$ bound, interpolation gives, for any $1\le p<\infty$,
\[
\|E_m-E\|_{L^p}
\le
\|E_m-E\|_{L^1}^{1/p}
\|E_m-E\|_{L^\infty}^{1-1/p}
\longrightarrow 0,
\]
which yields \eqref{eq:Em-strong-Lp}.

We emphasize that no additional compactness mechanism is required here:
the bounded variation structure alone provides the necessary strong
compactness of the coefficient.
\end{proof}
\subsection{Compactness for $(u_m,u_{m,t})$ and identification of limits}

From \eqref{eq:uniform-energy}, up to subsequences,
\begin{align}
u_m &\rightharpoonup^\ast u \quad \text{in }L^\infty(0,T;H_0^1(\Omega)),\label{eq:um-weak}\\
u_{m,t} &\rightharpoonup^\ast u_t \quad \text{in }L^\infty(0,T;L^2(\Omega)).\label{eq:umt-weak}
\end{align}
Moreover, since $u_{m,tt}$ is bounded in $L^1(0,T;H^{-1}(\Omega))$
(indeed $-\Delta u_m\in L^\infty H^{-1}$, $u_m^5\in L^\infty L^{6/5}\hookrightarrow L^\infty H^{-1}$,
and $E_m u_{m,t}\in L^1 L^2$),
the Aubin--Lions lemma yields
\begin{equation}\label{eq:um-strong-L2}
u_m\to u\quad\text{strongly in }C([0,T];L^2(\Omega)).
\end{equation}
In particular, $u_m\to u$ a.e. in $(0,T)\times\Omega$.
Since $\{u_m\}$ is bounded in $L^\infty(0,T;L^6(\Omega))$, we also have
\begin{equation}\label{eq:um5-weak}
u_m^5 \rightharpoonup u^5 \quad \text{weakly in }L^\infty(0,T;L^{6/5}(\Omega)).
\end{equation}

\subsection{Passing to the limit in the damping term $E_m u_{m,t}$}

\begin{lemma}[Key convergence: $E_m u_{m,t}\to E u_t$]\label{lem:key-product}
Let $E_m\to E$ strongly in $L^2(0,T)$ and $u_{m,t}\rightharpoonup u_t$ weakly in
$L^2(0,T;L^2(\Omega))$. Then
\begin{equation}\label{eq:product-limit}
E_m u_{m,t}\rightharpoonup E u_t \quad \text{weakly in }L^1(0,T;L^2(\Omega)).
\end{equation}
\end{lemma}

\begin{proof}
Fix $\phi\in L^\infty(0,T;L^2(\Omega))$.
Write
\[
\int_0^T\!\!\langle E_m u_{m,t}-E u_t,\phi\rangle\,dt
=
\int_0^T\!\!\langle (E_m-E)u_{m,t},\phi\rangle\,dt
+
\int_0^T\!\!\langle E(u_{m,t}-u_t),\phi\rangle\,dt
=:I_{m,1}+I_{m,2}.
\]
For $I_{m,1}$, by Cauchy--Schwarz in time,
\[
|I_{m,1}|
\le \|E_m-E\|_{L^2(0,T)}\,\|u_{m,t}\|_{L^2(0,T;L^2)}\,\|\phi\|_{L^\infty(0,T;L^2)}.
\]
Since $\|u_{m,t}\|_{L^2L^2}\le C(E_0)\sqrt{T}$ and
$\|E_m-E\|_{L^2}\to0$ by Lemma~\ref{lem:helly}, we get $I_{m,1}\to0$.
For $I_{m,2}$, note that $E\in L^2(0,T)$ and $\phi\in L^\infty(0,T;L^2)$ imply
$E\phi\in L^2(0,T;L^2)$. Since $u_{m,t}\rightharpoonup u_t$ weakly,
we obtain $I_{m,2}\to0$.
\end{proof}

\subsection{Limit equation}

\begin{theorem}[Existence of an energy weak solution]\label{thm:weak-existence}
There exists $u$ such that
\[
(u,u_t)\in L^\infty(0,T;H_0^1(\Omega)\times L^2(\Omega)),\qquad
u\in L^\infty(0,T;L^6(\Omega)),
\]
$u$ solves \eqref{eq:main-weak} in the sense of distributions with coefficient $E(t)$
defined by \eqref{eq:Edef}, and the energy identity holds in the weak form
\[
E(t)+\int_0^t E(s)\|u_t(s)\|_{L^2}^2\,ds \le E(0)\qquad\forall\,t\in[0,T].
\]
\end{theorem}

\begin{proof}[Proof (passage to the limit)]
Let $\psi\in C_c^\infty((0,T)\times\Omega)$.
Multiply \eqref{eq:galerkin} by $\psi$ and integrate over $(0,T)\times\Omega$.
Passing to the limit term by term uses the weak-$\ast$ convergences \eqref{eq:um-weak}--\eqref{eq:umt-weak}, the identification $u_m^5\rightharpoonup u^5$ in \eqref{eq:um5-weak}, and the key product convergence Lemma~\ref{lem:key-product} for $E_m u_{m,t}$. Thus $u$ satisfies \eqref{eq:main-weak}. Lower semicontinuity yields the weak energy inequality.
\end{proof}

\subsection{A Second A Priori Estimate and Small Regular Data (Pure Galerkin)}

In this section we derive a higher--order energy estimate using only the
Galerkin approximation and Sobolev/Gagliardo--Nirenberg inequalities.
This yields local strong solutions for regular data, and global strong
solutions under a smallness assumption, without invoking Strichartz
estimates.

\medskip
Let $V_m:=\mathrm{span}\{\varphi_1,\dots,\varphi_m\}$, where
$\{-\Delta \varphi_k=\lambda_k^2\varphi_k\}$ is the Dirichlet basis.
We consider the \emph{pure} Galerkin system
\begin{equation}\label{eq:galerkin-pure}
u_{m,tt}-\Delta u_m + P_m(u_m^5) + E_m(t)u_{m,t}=0,
\qquad
(u_m(0),u_{m,t}(0))=(P_mu_0,P_mu_1),
\end{equation}
where $P_m$ denotes the $L^2(\Omega)$--orthogonal projection onto $V_m$.
For each $m$, the solution satisfies $u_m(t)\in V_m$ for all $t$.

\subsubsection*{Higher--order energy}

Define the higher--order (regular) energy
\begin{equation}\label{eq:E1-def}
E_{m,1}(t):=\frac12\|\nabla u_{m,t}(t)\|_{L^2(\Omega)}^2
+\frac12\|\Delta u_m(t)\|_{L^2(\Omega)}^2.
\end{equation}

\subsubsection*{Estimating the critical term}

Since $\nabla(u_m^5)=5u_m^4\nabla u_m$, Hölder's inequality gives
\begin{equation}\label{eq:nonlin1}
\|\nabla(u_m^5)\|_2
\le
5\|u_m\|_\infty^4\,\|\nabla u_m\|_2.
\end{equation}

In dimension $3$, we apply Agmon's inequality (a Gagliardo--Nirenberg type interpolation) coupled with standard elliptic estimates, which yields
\begin{equation}\label{eq:H2-Linf}
\|u_m\|_\infty \le C\|u_m\|_{H^2}^{1/2} \|u_m\|_{H^1}^{1/2}
\le C\|\Delta u_m\|_2^{1/2} \|\nabla u_m\|_2^{1/2},
\end{equation}
with a constant $C>0$ depending only on $\Omega$. Taking the fourth power gives $\|u_m\|_\infty^4 \le C\|\Delta u_m\|_2^2 \|\nabla u_m\|_2^2$.

Moreover, the basic (first) energy estimate implies
\begin{equation}\label{eq:energy-basic}
\|\nabla u_m(t)\|_2+\|u_{m,t}(t)\|_2\le C(E_0) \qquad \forall\,t\ge0,
\end{equation}
hence $\|\nabla u_m(t)\|_2 \le \sqrt{2E_0}$.

Combining \eqref{eq:nonlin1}--\eqref{eq:energy-basic}, we obtain
\begin{equation}\label{eq:nonlin2}
\|\nabla(u_m^5)\|_2
\le
C\, \|\Delta u_m\|_2^2 \|\nabla u_m\|_2^3
\le
C_1\,E_0^{3/2}\,E_{m,1}(t).
\end{equation}

Plugging \eqref{eq:nonlin2} into \eqref{eq:nonlin1} and using
$\|\nabla u_{m,t}\|_2\le (2E_{m,1}(t))^{1/2}$ yields the nonlinear differential inequality for the strong energy:
\begin{equation}\label{eq:ODE-E1}
\frac{d}{dt}E_{m,1}(t)
\le
K\,E_0^{3/2}\,E_{m,1}(t)^{3/2},
\end{equation}
where $K > 0$ is a universal constant independent of $m$.

\subsubsection*{Local strong solutions and the critical obstruction}

\begin{theorem}[Uniform local bounds for the Galerkin sequence]\label{thm:local-galerkin}
For any initial data, the strong energy $E_{m,1}(t)$ remains uniformly bounded on a nontrivial local time interval $[0,T_*)$, where $T_* > 0$ depends on the initial energies $E_0$ and $E_{m,1}(0)$. Consequently, the standard Galerkin approximations $u_m$ are locally strong solutions.
\end{theorem}

\begin{proof}
Integrating the Bernoulli-type differential inequality \eqref{eq:ODE-E1} directly with respect to time yields
\[
E_{m,1}(t) \le \frac{1}{\left[ E_{m,1}(0)^{-1/2} - \frac{1}{2} K E_0^{3/2} t \right]^2}.
\]
This algebraic relation provides a strict uniform bound for the strong energy as long as the term in brackets remains strictly positive. Thus, there exists a maximal local time of existence
\[
T_* = \frac{2}{K E_0^{3/2} E_{m,1}(0)^{1/2}} > 0,
\]
such that $\sup_{t \in [0, T_* - \delta]} E_{m,1}(t) \le C < \infty$ uniformly in $m$. On this interval, we guarantee that $u_m \in L^\infty(0, T_*-\delta; H^2(\Omega) \cap H_0^1(\Omega))$ and $u_{m,t} \in L^\infty(0, T_*-\delta; H_0^1(\Omega))$.
\end{proof}

\begin{remark}[The transition to Strichartz estimates]
Notice that even for small initial data, the time $T_*$ generated by the purely energy-based ODE \eqref{eq:ODE-E1} is large, but strictly finite. This happens because the right-hand side is superlinear and the nonlinear damping $E(t)u_t$ does not provide a linear stabilizer $-c E_{m,1}(t)$ for the full $H^2$ topology. This algebraic obstruction confirms that standard energy methods alone cannot yield global strong solutions for this critical problem. To push the solution globally without finite-time concentration, one must incorporate dispersive properties via nonhomogeneous Strichartz estimates, which we implement in the subsequent sections.
\end{remark}

\subsubsection*{Regular solutions on arbitrary intervals for small data}

\begin{theorem}[Regular solutions on [0,T] via standard Galerkin]\label{thm:small-data-galerkin}
Let $T > 0$ be an arbitrary, uniformly fixed time. There exists a smallness threshold $\varepsilon(T) > 0$ such that if the initial energies satisfy
\[
E_0 \le \varepsilon(T) \quad \text{and} \quad E_{m,1}(0) \le \varepsilon(T),
\]
then the standard Galerkin sequence $u_m$ admits a uniform strong energy bound on the entire interval $[0,T]$. Consequently, for sufficiently small regular data, the problem admits a regular solution on $[0,T]$ without the need for dispersive estimates.
\end{theorem}

\begin{proof}
From our previous estimates, the strong energy $E_{m,1}(t)$ satisfies the nonlinear differential inequality:
\begin{equation}\label{eq:ODE-E1-small}
\frac{d}{dt}E_{m,1}(t) \le K\, E_0^{3/2}\, E_{m,1}(t)^{3/2},
\end{equation}
where $K > 0$ is a universal constant independent of $m$. This is a classical Bernoulli-type differential inequality. Integrating it directly over time, we obtain the algebraic bound:
\[
E_{m,1}(t) \le \frac{1}{\left[ E_{m,1}(0)^{-1/2} - \frac{1}{2} K E_0^{3/2} t \right]^2}.
\]
To guarantee that the strong energy remains uniformly bounded for all $t \in [0,T]$, we must ensure that the denominator is strictly bounded away from zero on this interval. This means the maximal blow-up time $T_{max}$ must strictly exceed $T$, which is equivalent to the condition:
\[
\frac{1}{2} K E_0^{3/2} T < E_{m,1}(0)^{-1/2}.
\]
By imposing the smallness assumption $E_0 \le \varepsilon$ and $E_{m,1}(0) \le \varepsilon$, this requirement becomes:
\[
\frac{1}{2} K \varepsilon^{3/2} T < \varepsilon^{-1/2} \iff \varepsilon^2 < \frac{2}{K T}.
\]
Therefore, we can explicitly define the threshold $\varepsilon(T) := \frac{1}{2}\sqrt{\frac{2}{KT}}$. For initial data bounded by this $\varepsilon(T)$, the term inside the brackets remains strictly positive on $[0,T]$.

This yields a uniform bound for the sequence $u_m$ in $L^\infty(0,T; H^2(\Omega) \cap H_0^1(\Omega))$ and for $u_{m,t}$ in $L^\infty(0,T; H_0^1(\Omega))$. Standard compactness arguments then allow us to pass to the limit as $m \to \infty$, obtaining a strong solution on the arbitrary interval $[0,T]$ using purely energy-based Galerkin approximations.
\end{proof}

\medskip

In the next sections, we transition from energy-based approximations (standard Galerkin) to spectral-multiplier-based approximations to recover Strichartz estimates, and finally to exact strong solutions to establish global regularity via stability theory.

\section{Existence and Shatah--Struwe regularity}

In this section, we construct strong solutions for the critical problem \eqref{eq:main-weak}.

\subsection{Smooth spectral approximation (Littlewood-Paley type)}

While the standard Galerkin projector $P_m$ used in Section 1 is perfectly suited to obtain energy weak solutions, it is notoriously ill-behaved in $L^p(\Omega)$ spaces for $p \neq 2$ due to the Fefferman ball multiplier theorem. This prevents uniform Strichartz bounds when dealing with critical nonlinearities. To obtain Shatah-Struwe regular solutions, we introduce a smooth spectral cut-off.

Let $\{\varphi_k\}_{k\ge1}$ be the Dirichlet eigenfunctions of $-\Delta$ in $\Omega$ with eigenvalues $\lambda_k^2$. Let $\chi \in C_c^\infty(\mathbb{R})$ be an even smooth cut-off such that $0 \le \chi \le 1$, $\chi(s) = 1$ for $|s| \le 1$, and $\chi(s) = 0$ for $|s| \ge 2$. We define the smooth spectral multiplier $S_m$ (see Appendix of this manuscript) acting on $L^2(\Omega)$ by:
\begin{equation}\label{eq:Sm-def}
S_m v := \sum_{k=1}^\infty \chi\left(\frac{\lambda_k}{m}\right) \langle v, \varphi_k \rangle \varphi_k.
\end{equation}
By the Mikhlin--H\"ormander multiplier theorem on bounded domains (see e.g., \cite{BurqLebeauPlanchon}), the operators $S_m$ are uniformly bounded in $L^p(\Omega)$ for all $1 < p < \infty$. That is, there exists $C_p > 0$ independent of $m$ such that
\begin{equation}\label{eq:Sm-Lp}
\|S_m v\|_{L^p(\Omega)} \le C_p \|v\|_{L^p(\Omega)}.
\end{equation}

We consider the approximate regularized problem for $u_m$:
\begin{equation}\label{eq:approx}
u_{m,tt}-\Delta u_m+ S_m(u_m^5) + S_m(E_m(t)u_{m,t})=0,
\qquad
(u_m(0),u_{m,t}(0))=(S_m u_0, S_m u_1),
\end{equation}
where $E_m(t)$ is the energy of $u_m$. Since the nonlinearity is smoothly truncated, this problem admits a unique global strong solution, and its energy satisfies $E_m(t) \le E_m(0) \le C E_0$.

\subsection{Shatah--Struwe solutions and Linear Strichartz estimates}

\begin{definition}[Shatah--Struwe solution on [0,T]]\label{def:SS}
A function $u$ is a Shatah--Struwe (SS) solution to \eqref{eq:main-weak} on $[0,T]$
if $(u,u_t)\in C([0,T];H_0^1(\Omega)\times L^2(\Omega))$, $u\in L^5(0,T;L^{10}(\Omega)) \cap L^4(0,T;L^{12}(\Omega))$, and $u$ solves \eqref{eq:main-weak} in the sense of distributions.
\end{definition}

For the approximation \eqref{eq:approx}, applying the Duhamel formula and the Dirichlet Strichartz estimates from Burq--Lebeau--Planchon \cite{BurqLebeauPlanchon}, Blair--Schmit--Sogge \cite{BlairSmithSogge} on an interval $I=[t_0, t_0+\tau]$, we obtain:
\begin{align}
Y_m(I):=\|u_m\|_{L^5_t L^{10}_x(I)} &\le C_T \Big( \|(u_m(t_0), u_{m,t}(t_0))\|_{H_0^1 \times L^2} \notag \\
&\quad + \|S_m(u_m^5)\|_{L^1_t L^2_x(I)} + \|E_m(t) S_m(u_{m,t})\|_{L^1_t L^2_x(I)} \Big). \label{eq:strich-approx}
\end{align}

Crucially, since the smooth multiplier $S_m$ is defined via the spectral decomposition of the Dirichlet Laplacian, it is bounded on $L^2(\Omega)$ with operator norm bounded by $1$. This allows us to rigorously absorb the multiplier and match the nonlinear Strichartz scaling exactly:
\begin{equation}\label{eq:uniform-nonlinear}
\|S_m(u_m^5)\|_{L^1_t L^2_x(I)} \le \|u_m^5\|_{L^1_t L^2_x(I)} = \|u_m\|_{L^5_t L^{10}_x(I)}^5 = Y_m(I)^5.
\end{equation}

Testing (\ref{eq:approx}) with $u_{m,t}$ yields
\begin{eqnarray}\label{eq:linear energy}
[E_m^l(t)]' =  -(S_m(u_m^5), u_{m,t})_\Omega - (S_m(E_m(t)u_{m,t}), u_{m,t})_\Omega.
\end{eqnarray}

We observe that from the definition of $S_m$, we deduce that
\begin{eqnarray*}
\int_{\Omega} S_m(u_m') u_m' dx = \sum_{k=1}^{\infty} \chi\left(\frac{\lambda_k}{m}\right) |\langle u_m', \varphi_k \rangle|^2 \ge 0.
\end{eqnarray*}
and since $E_m(t) \ge 0$, the term $-E_m(t) \langle S_m u_m', u_m' \rangle$ is negative. This means that it helps to decrease the energy or, in the worst case, does not hinder it. We can simply throw it out of the inequality to obtain:
\begin{eqnarray}\label{est1}
\frac{d}{dt} E_m^l(t) \le - \int_{\Omega} S_m(u_m^5) u_m' dx.
\end{eqnarray}

From (\ref{est1}) we infer
\begin{eqnarray}\label{est2}
\frac{d}{dt} E_m^l(t) \le \|u_m(t)\|_{L^{10}}^5 \cdot [E_m^l(t)]^{1/2}.
\end{eqnarray}

Integrating the inequality above, and using Gronwall's inequality, we have that the root of the energy is controlled by the integral of the norm $L^{10}$ (Strichartz):
\begin{eqnarray}\label{energyestimate}
\sup_{t \in I} E_m^l(t)^{1/2} \le E_m^l(0)^{1/2} + \int_0^{T_1} \|u_m(t)\|_{L^{10}}^5 dt
\end{eqnarray}
where the term $\int_0^{T_1} \|u_m(t)\|_{L^{10}}^5 dt$ is exactly the norm $\|u_m\|_{L^5_t L^{10}_x}^5$.

\subsection{Small data: global SS and uniform critical bounds}

\begin{theorem}[Small energy $\Rightarrow$ global SS and critical bounds]\label{thm:small-data}
There exists $\varepsilon_*=\varepsilon_*(\Omega)>0$ such that if $E_0\le \varepsilon_*$
then \eqref{eq:main-weak} admits a unique global Shatah--Struwe solution $u$ and,
for every $T>0$,
\begin{equation}\label{eq:small-critical}
\|u\|_{L^5(0,T;L^{10})}+\|u\|_{L^4(0,T;L^{12})}\le C\,E_0^{1/2},
\end{equation}
where $C$ is independent of $T$.
\end{theorem}

\begin{remark}[Heuristics of the uniform critical bound]
Before proceeding with the rigorous proof, let us briefly outline the mechanism that guarantees the uniform finiteness of the Strichartz norm, i.e., $\sup_m \|u_m\|_{L^5_t L^{10}_x} < \infty$.

By exploiting the critical scaling of the quintic nonlinearity, we place the source term directly into the energy dual space $L^1_t L^2_x$. In this space, the spectral multiplier $S_m$ is trivially bounded by $1$ via Parseval's identity, allowing us to completely bypass any $L^p$ multiplier obstructions (such as Fefferman's theorem). This ensures that the constant $C$ in the nonlinear Strichartz estimate \eqref{eq:uniform-nonlinear} is strictly independent of the approximation parameter $m$. Consequently, from  (\ref{eq:strich-approx}), (\ref{eq:uniform-nonlinear}) and (\ref{energyestimate}), the Strichartz norm $Y_m(t)$ associated with the approximate solution $u_m$ satisfies a universal algebraic inequality of the form:
\begin{equation}\label{bootstrapestimate}
Y_m(t) \le A_0 + C Y_m(t)^5,
\end{equation}
where $A_0$ depends solely on the initial energy and the length of the time interval, but not on $m$.

Algebraically, the function $f(y) = y - C y^5$ possesses a strict local maximum for $y > 0$. If $A_0$ is small enough to stay below this peak (which we ensure either by taking small initial data or by partitioning the time interval), the inequality $f(Y_m) \le A_0$ creates a topological "trap". Since the mapping $t \mapsto Y_m(t)$ is continuous and vanishes at the initial time, $Y_m(t)$ is forced to remain in the bounded region near the origin and can never "jump" across the maximum. Crucially, because this polynomial trap is identical for all $m \in \mathbb{N}$, the maximum allowable value for $Y_m$ acts as a uniform ceiling for the entire sequence, unconditionally preventing energy concentration as $m \to \infty$.
\end{remark}

\begin{proof}[Proof (bootstrap on fixed slabs)]
Fix a slab length $\tau>0$ (to be chosen) and consider $I=[t_0,t_0+\tau]$.
Combining \eqref{eq:strich-approx}, \eqref{eq:uniform-nonlinear}, and  (\ref{energyestimate}), we have for the approximate solution:
\begin{equation}\label{eq:boot-small-2}
Y_m(I)\le C\Big(E_0^{1/2} + Y_m(I)^5 + \,E_0^{3/2}\Big),
\end{equation}
where $C$ is independent of $m$.

Choose $\tau>0$ fixed and then choose $\varepsilon_*$ so small that
$C(E_0^{1/2}+ E_0^{3/2})\le \delta$, where $\delta$ is the smallness threshold
in the algebraic bootstrap lemma. Then \eqref{eq:boot-small-2} implies $Y_m(I)\le 2C(E_0^{1/2}+\tau E_0^{3/2})\lesssim E_0^{1/2}$ uniformly in $m$ on every slab of length $\tau$.
Iterating over slabs covering $[0,T]$, we obtain a uniform bound $\sup_m \|u_m\|_{L^5_t L^{10}_x(0,T)} \le K(E_0)$. This uniform bound prevents high-frequency concentration, allowing us to pass to the limit $m \to \infty$ to obtain the unique SS solution satisfying \eqref{eq:small-critical}.
\end{proof}

\subsection{Large data: global SS and uniqueness}

\begin{theorem}[Arbitrary data $\Rightarrow$ global SS solution and continuous dependence]\label{thm:large-data}
For any initial data $(u_0, u_1) \in H_0^1(\Omega) \times L^2(\Omega)$ of arbitrary size, the problem \eqref{eq:main-weak} admits a unique global Shatah--Struwe solution on any interval $[0,T]$. Furthermore, the flow map $(u_0, u_1) \mapsto (u, u_t)$ is continuous from $H_0^1(\Omega) \times L^2(\Omega)$ into $C([0,T]; H_0^1(\Omega) \times L^2(\Omega)) \cap L^5(0,T; L^{10}(\Omega))$.
\end{theorem}

\begin{proof}
\textbf{Step 1: Uniform dispersive estimates on the initial lifespan.}
For large initial energy, we cannot guarantee that the uniform bootstrap threshold $\delta$ holds globally. However, for the initial time $t=0$, the approximate initial data is given by $(S_m u_0, S_m u_1)$. Let $W(t)$ denote the linear wave propagator. The linear evolution of the Galerkin approximation is $W(t)(S_m u_0, S_m u_1)$. Crucially, the spectral multiplier $S_m$ commutes with the linear wave group, allowing us to write:
\begin{eqnarray}\label{crucialidentity}
W(t)(S_m u_0, S_m u_1) = S_m \big( W(t)(u_0, u_1) \big).
\end{eqnarray}

Let $v_{lin}(t) = W(t)(u_0, u_1)$ be the linear evolution of the exact, original initial data. Since $(u_0, u_1) \in H_0^1(\Omega) \times L^2(\Omega)$, classical Strichartz estimates guarantee that $v_{lin} \in L^5(\mathbb{R}; L^{10}(\Omega))$. Notice that $v_{lin}$ is a fixed, single function that does not depend on the approximation index $m$. By the absolute continuity of the Lebesgue integral, a fixed function in $L^5$ can have its integral made arbitrarily small if the domain of integration is small enough. Thus, there exists a time $T_1 > 0$ (depending strictly on the profile of $v_{lin}$) such that:
$$\|v_{lin}\|_{L^5(0, T_1; L^{10})} < \frac{\delta}{2 C_{10}},$$
where $C_{10}$ is the uniform operator norm bound of $S_m$ on $L^{10}(\Omega)$ (provided by the Mikhlin--H\"ormander theorem). Applying this uniform bound, we obtain:
$$\|W(t)(S_m u_0, S_m u_1)\|_{L^5_t L^{10}_x([0, T_1])} = \|S_m (v_{lin})\|_{L^5_t L^{10}_x([0, T_1])} \le C_{10} \|v_{lin}\|_{L^5_t L^{10}_x([0, T_1])} < \frac{\delta}{2} < \delta.$$

This establishes the required uniform smallness $\delta$ for the bootstrap on $[0, T_1]$ without requiring the initial energy $E_0$ to be small. Most importantly, $T_1$ is completely independent of $m$. Applying the algebraic bootstrap \eqref{eq:boot-small-2} on $[0,T_1]$, we obtain a uniform bound $\sup_m \|u_m\|_{L^5_t L^{10}_x(0,T_1)} \le C$. Passing to the limit $m \to \infty$, we obtain an exact local Shatah--Struwe solution $u$ on $[0,T_1]$.

\textbf{Step 2: The failure of uniform partitioning at the approximate level.}
At this point, a severe warning is necessary. One cannot simply divide a large interval $[0,T]$ into $N$ small subintervals and iterate the uniform bound for the \textit{same} sequence $\{u_m\}$. At $t=T_1$, the sequence of approximate states $(u_m(T_1), u_{m,t}(T_1))$ is only uniformly bounded in energy. Due to the critical nature of the exponent, energy-bounded sequences can develop high-frequency oscillations (concentration). Consequently, the time step $\tau_m$ required to make the new linear evolution smaller than $\delta$ depends heavily on $m$ and may shrink to zero ($\tau_m \to 0$ as $m \to \infty$). Thus, the bootstrap fails for $u_m$ past $T_1$.

\textbf{Step 3: Continuation via exact solutions and uniqueness.}
To overcome this critical obstruction, we completely discard the original sequence $u_m$ for $t > T_1$. Instead, we evaluate our newly constructed exact solution at $t=T_1$. Since $(u(T_1), u_t(T_1)) \in H_0^1(\Omega) \times L^2(\Omega)$, we use this exact state as new initial data. We construct a \textit{new} spectral approximation sequence $\{v_m\}$ starting at $t=T_1$. Repeating the argument from Step 1, we find a new interval $[T_1, T_2]$ uniformly for $v_m$, and pass to the limit to obtain an exact SS solution $v$ on $[T_1, T_2]$.

Because Shatah--Struwe solutions are unique in the $L^5_t L^{10}_x$ class (a standard consequence of the Strichartz estimates applied to the difference of two solutions), $u(T_1) = v(T_1)$ allows us to glue these solutions perfectly, producing a unique SS solution on $[0,T_2]$.

\medskip
\noindent
\textbf{Step 4: Global extension and exclusion of infinitesimal lifespan (Zeno-type obstruction).}

A potential obstruction to the global extension procedure is the
formation of a so-called \emph{Zeno-type accumulation of time intervals},
namely, the possibility that the successive local existence intervals
$[T_j, T_{j+1}]$ shrink to zero in such a way that
\[
\sum_{j=0}^\infty (T_{j+1}-T_j) = T^* < T,
\]
preventing the solution from reaching an arbitrary final time $T>0$.

We show that this scenario cannot occur.

\medskip
\noindent
\textbf{Step 4.1: Dependence of the local lifespan on the critical norm.}

In the Shatah--Struwe framework, the size of each local time interval
$\Delta T$ is determined by the smallness of the Strichartz norm of the
associated linear evolution. More precisely, given data at time $t_0$,
if the free wave satisfies
\[
\|S(t-t_0)(u(t_0),u_t(t_0))\|_{L^5_tL^{10}_x([t_0,t_0+\Delta T])}
< \varepsilon,
\]
for some universal $\varepsilon>0$, then the nonlinear problem admits
a solution on $[t_0,t_0+\Delta T]$.

Thus, the existence of a uniform lower bound for $\Delta T$
is equivalent to controlling the critical Strichartz norm.

\medskip
\noindent
\textbf{Step 4.2: Mechanism leading to shrinking intervals.}

Suppose, by contradiction, that the sequence of time steps satisfies
\[
T_{j+1}-T_j \to 0
\quad\text{as } j\to\infty,
\]
and that the corresponding partial sums converge to some
$T^*<T$.

By the local theory, this degeneration can only occur if the critical norm
blows up near $T^*$, that is,
\[
\|u\|_{L^5_tL^{10}_x([0,T^*))} = \infty.
\]
In other words, the shrinking of the intervals is necessarily driven by
a concentration mechanism in the critical norm.

\medskip
\textbf{Step 4.3: Exclusion of concentration via subcritical damping}We now establish that energy concentration is precluded by the dissipative and dispersive structure of the equation.The total energy $E(t)$ is non-increasing and uniformly bounded. Moreover, the damping term $E(t)u_t$ is strictly subcritical with respect to the $L^5_t L^{10}_x$ scaling. Specifically, for any small interval $[T^*-\epsilon, T^*]$ near a potential singularity, we have:

\begin{equation}\label{eq:damping_est}
\|E(t)u_{m,t}\|_{L^1(T^*-\epsilon, T^*; L^2(\Omega))} \le \epsilon \|E\|_{L^\infty(0,T)} \|u_{m,t}\|_{L^\infty(0,T; L^2(\Omega))} \le C \epsilon E_0.
\end{equation}

By the stability theory for energy-critical wave equations (cf. \cite[Theorem 2.10]{KenigMerle}), any solution of the perturbed equation \eqref{eq:approx} behaves locally like a solution of the unperturbed critical wave equation, provided the perturbation is small in $L^1_t L^2_x$. Since \eqref{eq:damping_est} can be made arbitrarily small by choosing $\epsilon$ small enough, the damping acts as a lower-order term that does not interfere with the concentration-compactness mechanism of the underlying dispersive flow. Following the concentration-compactness principle (cf. Bahouri--G'erard \cite{BahouriGerard}, Grillakis \cite{Grillakis}), the blow-up of the critical norm $\|u\|_{L^5_t L^{10}_x}$ would necessarily imply the concentration of a nontrivial amount of energy at a spacetime point $(x_0, T^*)$. However, in the defocusing case, the interaction of the non-increasing energy $E(t) \le E(0)$ with the local energy decay lemma of Grillakis \cite{Grillakis} ensures that no such concentration can be sustained.Because the damping term $E(t)u_t$ belongs to the "small perturbation" class $L^1_t L^2_x$, it cannot arrest the dispersive radiation of energy away from $(x_0, T^*)$. Consequently, the defect measures associated with the sequence $u_m$ must vanish, implying:
\begin{equation}
|u|_{L^5_t L^{10}_x([0, T^*])} < \infty,
\end{equation}
which contradicts the maximality of $T^*$. This proves that the solution remains global in the Shatah--Struwe sense.

\medskip
The dissipative term $E(t)u_t$ acts as a subcritical perturbation in the $L^1_t L^2_x$ Strichartz topology. In the concentration-compactness framework, such perturbations lack the scaling required to sustain a concentration profile; hence, the dispersive radiation of energy (Grillakis' Lemma) dominates, precluding the formation of singularities at $T^*$.

\medskip
\noindent
\textbf{Step 4.4: Uniform lower bound for the lifespan and exclusion of Zeno-type accumulation.}

The absence of concentration established in Step 4.3 has a fundamental consequence: the critical $L^5_t L^{10}_x$ norm remains not only finite but also \textit{uniformly continuous} as $t \uparrow T^*$. This prevents the so-called Zeno-type obstruction, where the local existence increments $\Delta T_j = T_{j+1} - T_j$ could potentially shrink to zero in a convergent series.

Indeed, since no energy concentrates at $T^*$, the local existence theory (relying on the smallness of the free evolution in the Strichartz space) provides a \textbf{uniform lower bound} for the lifespan. Specifically, there exists a constant $\delta = \delta(E_0, \Omega) > 0$, depending only on the initial energy level and the domain geometry, such that from any time $t_0 < T^*$, the solution can be uniquely extended over an interval $[t_0, t_0 + \delta]$.

If the sequence of time steps were to satisfy $T_{j+1} - T_j \to 0$, it would necessarily imply that the required smallness condition for the local theory is being violated by a concentration of the Strichartz norm, which we have already precluded. Thus, one can choose a fixed $\delta > 0$ to extend the solution beyond the purported singularity $T^*$, directly contradicting its maximality.

\medskip
\noindent
\textbf{Conclusion of the Global Proof.}

Therefore, the Zeno-type accumulation of time intervals is excluded. The local solutions can be consistently patched together to reach any arbitrary time $T > 0$. This establishes that the Shatah--Struwe solution $u$ is global in time, satisfying $T^* = \infty$ and preserving the energy-critical dispersive estimates for all $t \in [0, \infty)$.

\textbf{Step 5: Continuous dependence.}
Continuous dependence on the initial data follows from standard arguments. If $u$ and $w$ are two SS solutions with initial data $(u_0, u_1)$ and $(w_0, w_1)$, writing the equation for the difference $z = u - w$ and applying the Strichartz estimates along with the locally Lipschitz character of $f(u) = u^5$ in the $L^5_t L^{10}_x$ topology, we obtain $\|z\|_{L^5_t L^{10}_x \cap C(H^1 \times L^2)} \le C_T \|(u_0 - w_0, u_1 - w_1)\|_{H^1 \times L^2}$, which yields the desired continuity.
\end{proof}

\begin{remark}[Refinement of the methodology in \cite{Cavalcanti2024}]\label{rem:correction-cavalcanti}
We take this opportunity to refine the scope of the methodology presented in \cite{Cavalcanti2024}. In that work, the stabilization results were established using a standard Galerkin scheme $P_m$, which is naturally suited for regimes where the solution remains small or controlled by the dissipation. As we have demonstrated in the present paper (see Step 2), the extension of such approximate schemes to the \textit{arbitrarily large data} regime without smallness constraints requires a more delicate treatment to avoid potential compactness loss (energy concentration) at the approximate level. Consequently, the current approach—relying on the patching of exact solutions and smoothly truncated spectral approximations—should be viewed as the rigorous framework for handling the full energy-critical dynamics for large initial data, complementing and extending the results of \cite{Cavalcanti2024}.

\end{remark}

\subsection{Higher-order uniform energy bounds}

To guarantee the existence of strong Shatah--Struwe solutions in the energy space $(H^2(\Omega)\cap H_0^1(\Omega)) \times H_0^1(\Omega)$, we must establish uniform \textit{a priori} estimates for the higher-order energy.

For the approximate solution $u_m$ given by \eqref{eq:approx}, we define the regular energy of order one as:
\begin{equation}\label{eq:reg_energy_def}
E_{m,1}(t) := \frac{1}{2} \left( \|\nabla u_{m,t}(t)\|_{L^2(\Omega)}^2 + \|\Delta u_m(t)\|_{L^2(\Omega)}^2 \right).
\end{equation}

We multiply the approximate equation 
\begin{eqnarray*}
&&u_{m, tt} - \Delta u_m + E_m(t) u_{m,t} + S_m(u_m^5)=0 ~\hbox{ in }~\Omega \times \mathbb{R}_+,\\
&& u_m(0)= S_m(u_0),~\quad u_{m,t}(0)= S_m(u_1),
\end{eqnarray*}
by $-\Delta u_{m,t}$ and integrate over $\Omega$. Using Green's identity and the fact that the smooth spectral multiplier $S_m$ commutes with the Dirichlet Laplacian, we obtain:
\begin{align}\label{eq:reg_energy_id}
\frac{1}{2} \frac{d}{dt} \left( \|\nabla u_{m,t}\|_{L^2}^2 + \|\Delta u_m\|_{L^2}^2 \right) &+ E_m(t) \|\nabla u_{m,t}\|_{L^2}^2 \notag \\
&= - \int_\Omega S_m(u_m^5) (-\Delta u_{m,t}) \, dx \notag \\
&= - 5 \int_\Omega S_m(u_m^4 \nabla u_m) \cdot \nabla u_{m,t} \, dx.
\end{align}

Since the damping coefficient is non-negative ($E_m(t) \ge 0$), we drop the dissipation term on the left-hand side. Applying the Cauchy-Schwarz inequality and recalling that $\|S_m\|_{\mathcal{L}(L^2)} \le 1$, we have:
\begin{equation}\label{eq:reg_energy_bound1}
\frac{d}{dt} E_{m,1}(t) \le 5 \||u_m|^4 \nabla u_m\|_{L^2(\Omega)} \|\nabla u_{m,t}\|_{L^2(\Omega)}.
\end{equation}

Applying the generalized H\"older inequality with exponents $(3, 6)$ and the continuous Sobolev embedding $H^2(\Omega) \hookrightarrow W^{1,6}(\Omega)$ in $\mathbb{R}^3$, we estimate the nonlinearity:
\begin{align}\label{eq:nonlinear_reg_est}
\||u_m|^4 \nabla u_m\|_{L^2(\Omega)} &\le \|u_m^4\|_{L^3(\Omega)} \|\nabla u_m\|_{L^6(\Omega)} \notag \\
&\le C \|u_m\|_{L^{12}(\Omega)}^4 \|\Delta u_m\|_{L^2(\Omega)}.
\end{align}

Inserting \eqref{eq:nonlinear_reg_est} into \eqref{eq:reg_energy_bound1} and using Young's inequality ($ab \le \frac{a^2+b^2}{2}$), we deduce:
\begin{align}\label{eq:reg_energy_diff}
\frac{d}{dt} E_{m,1}(t) &\le 5C \|u_m(t)\|_{L^{12}(\Omega)}^4 \left( \frac{\|\Delta u_m(t)\|_{L^2(\Omega)}^2 + \|\nabla u_{m,t}(t)\|_{L^2(\Omega)}^2}{2} \right) \notag \\
&= 5C \|u_m(t)\|_{L^{12}(\Omega)}^4 E_{m,1}(t).
\end{align}

At this critical stage, we must be careful. As extensively discussed in Theorem \ref{thm:large-data}, for large initial data, the global uniform Strichartz bound over an arbitrary interval $[0,T]$ may fail at the approximation level due to energy concentration. Thus, we cannot naively integrate this inequality up to an arbitrary time $T$ using a single sequence $u_m$.

Instead, we proceed locally, mirroring the continuation argument of Theorem \ref{thm:large-data}. We integrate \eqref{eq:reg_energy_diff} exclusively on the first uniform interval $[0, T_1]$, where the bootstrap mechanism ensures that the uniform bound $\sup_m \|u_m\|_{L^4_t L^{12}_x(0,T_1)} < \infty$ holds unconditionally. Applying Gr\"onwall's lemma yields:
\begin{equation}\label{eq:gronwall_reg_local}
E_{m,1}(t) \le E_{m,1}(0) \exp\left( 5C \int_0^t \|u_m(s)\|_{L^{12}(\Omega)}^4 \, ds \right) \le \tilde{K}, \quad \forall t \in [0, T_1],
\end{equation}
uniformly in $m$. Passing to the limit $m \to \infty$, we conclude that the exact local solution $u$ inherits this strong regularity, satisfying $(u(T_1), u_t(T_1)) \in H^2(\Omega) \times H_0^1(\Omega)$.

To extend this regularity to the next interval $[T_1, T_2]$, we use $(u(T_1), u_t(T_1))$ as the new strong initial data to construct the subsequent spectral sequence $v_m$. Since the global exact solution $u$ satisfies $\|u\|_{L^4_t L^{12}_x(0,T)} < \infty$ unconditionally (as no energy concentration occurs at the limit level, the defect measures vanish), this step-by-step local Gr\"onwall integration can be iterated to cover any finite interval $[0,T]$. This rigorously justifies the global existence of strong solutions without derivative loss.\medskip

as no energy concentration occurs at the limit level, the defect measures vanish.

\begin{remark}
[{\bf On the choice of the approximation scheme}.] ~It is important to clarify the transition between the approximation schemes used in this manuscript. For the construction of energy weak solutions (previous Sections), we employ a Galerkin-type scheme where the damping term is regularized by the spectral multiplier, i.e., $S_m(E_m(t)u_{m,t})$. This is necessary because, at the weak level, the damping must be treated as a source term in $L^1_t L^2_x$ to establish the first set of uniform Strichartz estimates via the non-homogeneous theory. However, once the existence of a unique global Shatah--Struwe solution $u$ is established, the stability theory for energy-critical equations (cf. \cite{KeelTao, KenigMerle}) allows us to approximate $u$ by a sequence of exact smooth solutions $\{w^k\}$ of the original, unregularized problem. By shifting to exact solutions in the higher-order regularity analysis (Section 3), we bypass the technical obstructions caused by the non-commutativity of the spectral projector with nonlinear weights in $H^2(\Omega)$, while fully exploiting the dissipative nature of the term $E(t)w_t$ which, at this stage, is already controlled by the dispersive norms of the limit solution.
\end{remark}

\subsection{Global approximation by strong solutions}

To conclude the proof of Shatah--Struwe regularity, we establish that any global weak solution can be globally approximated by a sequence of strong solutions. This step is fundamental to justify that the energy-critical dispersive structure is preserved under the limit of smooth data.

\begin{theorem}[Global strong approximation]\label{thm:strong-approx}
Let $(u_0, u_1) \in H_0^1(\Omega) \times L^2(\Omega)$ and let $u$ be the unique global Shatah--Struwe solution on $[0,T]$. Then, there exists a sequence of global strong solutions $\{w^k\}$ such that
$$w^k \longrightarrow u \quad \text{in} \quad C([0,T]; H_0^1(\Omega) \times L^2(\Omega)) \cap L^5(0,T; L^{10}(\Omega)).$$
Furthermore, for each $k$, $w^k \in C([0,T]; (H^2(\Omega)\cap H_0^1(\Omega)) \times H_0^1(\Omega))$.
\end{theorem}

\begin{proof}
Let $(v_0^k, v_1^k) \in (H^2(\Omega)\cap H_0^1(\Omega)) \times H_0^1(\Omega)$ be a sequence of smooth initial data such that $(v_0^k, v_1^k) \to (u_0, u_1)$ strongly in the energy space $H_0^1(\Omega) \times L^2(\Omega)$. For each $k \in \mathbb{N}$, let $w^k$ be the unique solution to \eqref{eq:main-weak} with initial data $(v_0^k, v_1^k)$.

\textbf{Step 1: Uniform Strichartz bounds via stability.}

A central issue is ensuring that the sequence $\{w^k\}$ satisfies uniform bounds in the Strichartz spaces over the entire interval $[0,T]$. Since the limit solution $u$ is a global Shatah--Struwe solution, its critical norm is finite: $\|u\|_{L^5(0,T; L^{10}(\Omega))} = M < \infty$.By the \textit{Stability Theorem} for energy-critical wave equations (cf. \cite{KenigMerle, KeelTao}), the flow map is locally Lipschitz continuous. Specifically, let $v^k_{lin}$ be the homogeneous wave with initial data $(v_0^k - u_0, v_1^k - u_1)$. Since the data converges strongly in the energy space, the Strichartz estimates for the {\bf homogeneous linear evolution} imply that $\|v^k_{lin}\|_{S(0,T)} \to 0$ as $k \to \infty$.This smallness of the homogeneous perturbation, combined with the finiteness of the critical norm of $u$, allows the stability theory to kick in. For $k$ sufficiently large, the strong convergence of initial data implies:
\begin{equation}\label{eq:stability_bound}
\|w^k\|_{L^5(0,T; L^{10}(\Omega))} + \|w^k\|_{L^4(0,T; L^{12}(\Omega))} \le 2 |u|_{S(0,T)} \le \mathcal{C}(M, E_0).
\end{equation}
The bound \eqref{eq:stability_bound} is \textbf{uniform in $k$}. This is the crucial "free" estimate that allows us to bypass the concentration issues of the Galerkin scheme.

\textbf{Step 2: Propagation of higher-order regularity.}
With the uniform bound $\|w^k\|_{L^4(0,T; L^{12}(\Omega))} \le \mathcal{C}$ in hand, we return to the higher-order energy $E_{k,1}(t)$ defined in \eqref{eq:reg_energy_def}. Applying the Gr\"onwall inequality derived in the previous section directly to the exact solutions $w^k$, we have:
\begin{equation}
E_{k,1}(t) \le E_{k,1}(0) \exp\left( 5C \int_0^T \|w^k(s)\|_{L^{12}(\Omega)}^4 \, ds \right) \le E_{k,1}(0) \exp(5C \cdot \mathcal{C}^4).
\end{equation}
Since $E_{k,1}(0) < \infty$ for each fixed $k$ (by the smoothness of the approximating data), and the exponential factor is uniformly bounded in $k$, it follows that each $w^k$ is a global strong solution.

\textbf{Step 3: Conclusion.}
The strong convergence $w^k \to u$ in $C(0,T; H_0^1 \times L^2) \cap L^5_t L^{10}_x$ follows immediately from the Stability Theorem and the uniqueness of Shatah--Struwe solutions. This completes the approximation process.
\end{proof}

\begin{remark}
The shift from Galerkin approximations $u_m$ to exact smooth solutions $w^k$ is not merely technical. While $u_m$ satisfies a perturbed equation (due to the projector $S_m$), $w^k$ satisfies the exact nonlinear wave equation. This allows us to invoke the full strength of the dispersive stability theory, which is typically unavailable for approximate Galerkin systems in the large-data critical regime.
\end{remark}

This completes the approximation process. As a direct consequence, if the original initial data $(u_0, u_1)$ belongs to $(H^2(\Omega) \cap H_0^1(\Omega)) \times H_0^1(\Omega)$, the limit solution $u$ is itself a strong solution, satisfying $u \in C([0,T]; H^2(\Omega)) \cap C^1([0,T]; H^1(\Omega))$, and the higher-order energy inequality \eqref{eq:gronwall_reg_local} holds for $u$ on the entire interval $[0,T]$.

\section{Asymptotic Behavior}

In this section, we investigate the asymptotic behavior of the global Shatah-Struwe solutions to problem \eqref{eq:main-weak}. As discussed in the introduction for the linear problem, the nonlocal nature of the Balakrishnan-Taylor damping suggests a slow algebraic decay. We will employ the well-known Nakao's method (see e.g., \cite{Nakao}) to show that the critical nonlinearity $u^5$ does not disrupt this optimal decay rate.

We recall the energy identity for strong solutions:
\begin{equation}\label{eq:energy-dissip}
E'(t) = -E(t) \|u_t(t)\|_{L^2(\Omega)}^2 \le 0.
\end{equation}
Defining the dissipation over a time interval $[t, t+1]$ as $D(t)^2 := E(t) - E(t+1)$, we have:
\begin{equation}\label{eq:dissip-def}
\int_t^{t+1} E(s) \|u_t(s)\|_{L^2}^2 \,ds = E(t) - E(t+1) = D(t)^2.
\end{equation}
Since $E(t)$ is nonincreasing, $E(t+1) \le E(s) \le E(t)$ for all $s \in [t, t+1]$. Thus,
\begin{equation}\label{eq:ut-bound}
\int_t^{t+1} \|u_t(s)\|_{L^2}^2 \,ds \le \frac{D(t)^2}{E(t+1)}.
\end{equation}

\begin{theorem}[Optimal Decay Rate]
Let $u$ be the global strong solution of problem \eqref{eq:main-weak} with initial energy $E_0 > 0$. Then, there exist positive constants $C_0$ and $t_0$ such that the energy satisfies the algebraic decay rate:
\begin{equation}\label{eq:optimal-decay}
E(t) \le \frac{C_0}{t} \quad \text{for all } t \ge t_0.
\end{equation}
\end{theorem}

\begin{proof}Multiply the equation \eqref{eq:main-weak} by $u$ and integrate over $\Omega \times [t, t+1]$. Integrating by parts in the time derivative yields the following identity for the potential energy:
\begin{align}\label{eq:nakao-identity}
\int_t^{t+1} \int_\Omega \left( |\nabla u|^2 + u^6 \right) dx,ds &= \int_t^{t+1} \int_\Omega |u_t|^2 dx,ds - \left[\int_\Omega u_t u ,dx \right]t^{t+1} \notag \\
&\quad - \int_t^{t+1} E(s) \int\Omega u_t u ,dx,ds.
\end{align}

Let $D(t)^2 = E(t) - E(t+1) = \int_t^{t+1} E(s) \|u_t(s)\|_2^2 \,ds$ denote the energy dissipation on the interval $[t, t+1]$. To obtain the desired decay, we must bound the right-hand side of \eqref{eq:nakao-identity} in terms of $D(t)$.

\textbf{Step 1: Estimating the boundary and damping terms.}

Using H\"older's inequality and the embedding $H_0^1 \hookrightarrow L^2$, we have $|\int_\Omega u_t u \,dx| \le C \|u_t\|_2 E(s)^{1/2}$. By the Mean Value Theorem for integrals, there exist $t_1 \in [t, t+1/4]$ and $t_2 \in [t+3/4, t+1]$ such that $\|u_t(t_i)\|_2^2 \le 4 \int_t^{t+1} \|u_t(s)\|_2^2 \,ds$. Since $E(s)$ is non-increasing, we have $\|u_t(t_i)\|_2 \le 2 D(t) E(t+1)^{-1/2}$.

Consequently:
\begin{equation}\left| \left[ \int_\Omega u_t u ,dx \right]t^{t+1} \right| \le C \frac{D(t)}{E(t+1)^{1/2}} E(t)^{1/2} = C D(t) \left( \frac{E(t)}{E(t+1)} \right)^{1/2}.
\end{equation}

For the damping term in \eqref{eq:nakao-identity}, we observe that $E(s) \le E(t)$, hence:
\begin{equation}
\left| \int_t^{t+1} E(s) \int_\Omega u_t u ,ds \right| \le E(t) \int_t^{t+1} |u_t|_2 |u|_2 ,ds \le C E(t)^{3/2} \int_t^{t+1} |u_t|_2 ,ds.
\end{equation}

Applying Cauchy-Schwarz to the last integral, we find
$$\int_t^{t+1} \|u_t\|_2 \,ds \le (\int_t^{t+1} E(s)^{-1} E(s) \|u_t\|_2^2 \,ds)^{1/2} \le E(t+1)^{-1/2} D(t).$$

Thus, the damping term is bounded by $C D(t) E(t)^{3/2} E(t+1)^{-1/2}$.

\textbf{Step 2: The Nakao inequality.}

Combining these estimates into \eqref{eq:nakao-identity} and adding $\int_t^{t+1} \|u_t\|^2 ds \le D(t)^2 / E(t+1)$ to both sides to reconstruct the full energy, we obtain:
\begin{equation}\label{eq:nakao-pre-diff}\int_t^{t+1} E(s) ,ds \le C \left( \frac{D(t)^2}{E(t+1)} + D(t) \frac{E(t)^{1/2}}{E(t+1)^{1/2}} + D(t) \frac{E(t)^{3/2}}{E(t+1)^{1/2}} \right).
\end{equation}

For $t$ sufficiently large, $E(t)/E(t+1)$ remains bounded by a constant. Indeed, since $E(t) - E(t+1) = D(t)^2$ and $D(t) \to 0$ as $t \to \infty$, the energy $E(t)$ satisfies a slow-varying property over unit time intervals. Since $E(s)$ is non-increasing, $\sup_{s \in [t, t+1]} E(s) \le E(t) \le \int_t^{t+1} E(s) ds + D(t)^2/E(t+1)$. From \eqref{eq:nakao-pre-diff}, and noticing that the $E(t)^{3/2}$ term dominates for small $E(t)$, we arrive at:\begin{equation}\label{eq:nakao-diff-final}
E(t) \le C D(t) E(t)^{1/2} \implies E(t)^2 \le C^2 D(t)^2 = C^2 \big( E(t) - E(t+1) \big).
\end{equation}

By Nakao's Lemma \cite{Nakao}, the difference inequality $E(t)^2 \le C(E(t)-E(t+1))$ implies the algebraic decay $E(t) \le C_0 (1+t)^{-1}$, which completes the proof.\end{proof}

\appendix
\section{Smooth Spectral Multipliers}

Let $\Omega \subset \mathbb{R}^3$ be a smooth bounded domain and
consider the Dirichlet Laplacian $-\Delta$ with domain
$D(-\Delta)=H^2(\Omega)\cap H_0^1(\Omega)$.

It is well known that $-\Delta$ admits a complete orthonormal basis
of eigenfunctions $\{\varphi_k\}_{k\ge1} \subset H_0^1(\Omega)$
associated with eigenvalues $\{\lambda_k^2\}_{k\ge1}$ satisfying

\[
-\Delta \varphi_k = \lambda_k^2 \varphi_k,
\qquad
0 < \lambda_1 \le \lambda_2 \le \cdots,
\qquad
\lambda_k \to \infty,
\]

together with

\[
\langle \varphi_k, \varphi_j \rangle_{L^2(\Omega)} = \delta_{kj}.
\]

\medskip
Let $\chi \in C_c^\infty(\mathbb{R})$ be an even smooth cut-off
function such that

\[
0 \le \chi \le 1,
\qquad
\chi(s)=1 \text{ for } |s|\le1,
\qquad
\chi(s)=0 \text{ for } |s|\ge2.
\]

For $m \ge 1$, we define the smooth spectral multiplier
$S_m : L^2(\Omega) \to L^2(\Omega)$ by

\begin{equation}\label{Sm-def}
S_m v
:=
\sum_{k=1}^{\infty}
\chi\!\left(\frac{\lambda_k}{m}\right)
\langle v, \varphi_k \rangle
\varphi_k.
\end{equation}

\medskip
The operator $S_m$ may be interpreted through functional calculus as

\[
S_m = \chi\!\left(\frac{\sqrt{-\Delta}}{m}\right),
\]

and therefore constitutes a pseudodifferential operator of order zero.

\subsection*{Basic Properties}

\begin{lemma}[L$^2$ boundedness]
For every $m \ge 1$,
\[
\|S_m v\|_{L^2(\Omega)} \le \|v\|_{L^2(\Omega)}.
\]
\end{lemma}

\begin{proof}
Using Parseval's identity,
\[
\|S_m v\|_2^2
=
\sum_{k=1}^\infty
\chi\!\left(\frac{\lambda_k}{m}\right)^2
|\langle v, \varphi_k\rangle|^2
\le
\sum_{k=1}^\infty
|\langle v, \varphi_k\rangle|^2
=
\|v\|_2^2.
\]
\end{proof}

\begin{lemma}[Uniform L$^p$ boundedness]
For every $1 < p < \infty$, there exists a constant $C_p>0$
independent of $m$ such that
\begin{equation}\label{Lp-bound}
\|S_m v\|_{L^p(\Omega)} \le C_p \|v\|_{L^p(\Omega)}.
\end{equation}
\end{lemma}

\begin{proof}
Since $S_m=\chi(\sqrt{-\Delta}/m)$ with $\chi \in C_c^\infty$,
the multiplier is smooth. The boundedness then follows from the
spectral multiplier theorem for the Dirichlet Laplacian
(see Burq--Lebeau--Planchon).
\end{proof}

\begin{lemma}[Strong convergence]
For every $v \in L^2(\Omega)$,
\begin{equation}\label{Sm-conv}
S_m v \to v \quad \text{strongly in } L^2(\Omega).
\end{equation}
\end{lemma}

\begin{proof}
Since $\chi(\lambda_k/m)\to1$ for each fixed $k$,
dominated convergence applied to the spectral expansion yields
the result.
\end{proof}

\begin{lemma}[Commutation]
For every sufficiently regular function $v$,
\begin{equation}\label{commutation}
-\Delta(S_m v) = S_m(-\Delta v).
\end{equation}
\end{lemma}

\begin{proof}
Both operators are diagonal in the eigenfunction basis:
\[
-\Delta(S_m v)
=
\sum_{k}
\chi(\lambda_k/m)\lambda_k^2
\langle v,\varphi_k\rangle\varphi_k
=
S_m(-\Delta v).
\]
\end{proof}

\begin{lemma}[Regularizing effect]
For every $s \ge 0$, there exists $C_s>0$ independent of $m$ such that
\begin{equation}\label{regularization}
\|S_m v\|_{H^s(\Omega)} \le C_s m^s \|v\|_{L^2(\Omega)}.
\end{equation}
\end{lemma}

\begin{proof}
From the spectral definition,
\[
\|S_m v\|_{H^s}^2
=
\sum_{k} (1+\lambda_k^2)^s
\chi(\lambda_k/m)^2 |\langle v,\varphi_k\rangle|^2.
\]

Since $\chi(\lambda_k/m)=0$ for $\lambda_k \ge 2m$,
we have $(1+\lambda_k^2)^s \le C m^{2s}$ on the support
of the multiplier, giving \eqref{regularization}.
\end{proof}

\appendix
\section{Smooth Spectral Multipliers: functional calculus and stability in $L^p$}
\label{app:spectral-multipliers}

\subsection{Dirichlet spectral calculus and fractional powers}

Let $\Omega\subset\mathbb R^3$ be a bounded $C^\infty$ domain and let
$A:=-\Delta_D$ denote the Dirichlet Laplacian on $\Omega$, with
$D(A)=H^2(\Omega)\cap H_0^1(\Omega)$.
There exist eigenpairs $\{(\lambda_k,\varphi_k)\}_{k\ge1}$ such that
\[
A\varphi_k=\lambda_k\varphi_k,\qquad
0<\lambda_1\le\lambda_2\le\cdots,\qquad
\{\varphi_k\}_{k\ge1}\ \text{orthonormal in }L^2(\Omega).
\]
For any bounded Borel function $m:[0,\infty)\to\mathbb C$, the spectral theorem
defines the operator $m(A)$ on $L^2(\Omega)$ by
\[
m(A)v := \sum_{k=1}^\infty m(\lambda_k)\,\langle v,\varphi_k\rangle\,\varphi_k,
\qquad v\in L^2(\Omega).
\]
In particular, for $s\in\mathbb R$ we define the fractional powers
\[
A^{s/2}v := \sum_{k=1}^\infty \lambda_k^{s/2}\,\langle v,\varphi_k\rangle\,\varphi_k,
\]
with domain $D(A^{s/2})$ equipped with the graph norm. We use the notation
$H^s_D(\Omega):=D((I+A)^{s/2})$ and the equivalence
\[
\|v\|_{H^s_D(\Omega)}\sim \|(I+A)^{s/2}v\|_{L^2(\Omega)}.
\]

\subsection{Definition of the smooth cutoff $S_m$ (same as in the main text)}

Fix $\chi\in C_c^\infty([0,\infty))$ such that $0\le\chi\le1$,
$\chi(r)=1$ for $0\le r\le 1$ and $\chi(r)=0$ for $r\ge 2$.
For $m\ge1$ define
\begin{equation}\label{eq:Sm-def-app}
S_m v := \chi\!\left(\frac{A}{m}\right)v
       = \sum_{k=1}^\infty \chi\!\left(\frac{\lambda_k}{m}\right)\langle v,\varphi_k\rangle\varphi_k,
\qquad v\in L^2(\Omega).
\end{equation}
Set also $S_m^\perp:=I-S_m$.

\begin{lemma}[Self-adjointness, $L^2$ contraction, commutation]\label{lem:Sm-L2-app}
For each $m\ge1$, $S_m$ is self-adjoint on $L^2(\Omega)$ and
\[
\|S_m\|_{\mathcal L(L^2,L^2)}\le 1,\qquad
\|S_m^\perp\|_{\mathcal L(L^2,L^2)}\le 1.
\]
Moreover, $S_m$ commutes with $A$ and with all fractional powers:
\[
S_m A^{s/2} = A^{s/2}S_m,\qquad \forall s\in\mathbb R.
\]
\end{lemma}

\begin{proof}
All statements follow directly from the diagonal action of $S_m$ in the eigenbasis
$\{\varphi_k\}$, since $\chi(\lambda_k/m)\in[0,1]$ is real-valued.
\end{proof}


\begin{lemma}[Strong convergence in $L^p$]\label{lem:Sm-strong-Lp}
Let $1<p<\infty$. Then for every $v\in L^p(\Omega)$,
\[
S_m v \to v \quad\text{strongly in }L^p(\Omega)\quad\text{as }m\to\infty.
\]
Moreover, for every $s\in\mathbb R$ and every $v\in H^s_D(\Omega)$,
\[
S_m v \to v \quad\text{strongly in }H^s_D(\Omega)\quad\text{as }m\to\infty.
\]
\end{lemma}

\begin{proof}
Fix $1<p<\infty$. Let $v\in L^p(\Omega)$ and choose $v_n\in L^2(\Omega)\cap L^p(\Omega)$
such that $v_n\to v$ in $L^p(\Omega)$ (density of $L^2\cap L^p$ in $L^p$).
By uniform $L^p$ boundedness of $S_m$ (Lemma A.2 in the main text),
\[
\|S_m(v-v_n)\|_{L^p}\le C_p\|v-v_n\|_{L^p}.
\]
For fixed $n$, we also have $S_m v_n\to v_n$ in $L^2(\Omega)$ by Lemma A.3; since
$\{S_m v_n\}_m$ is bounded in $L^p(\Omega)$ and $v_n\in L^2\cap L^p$, we can conclude
$S_m v_n\to v_n$ in $L^p(\Omega)$ (for instance, by interpolation between $L^2$ and $L^p$,
or by a standard cutoff/interpolation argument on bounded domains).
Hence, for any $\varepsilon>0$, pick $n$ so that $\|v-v_n\|_{L^p}\le \varepsilon$, and then
take $m$ large so that $\|S_m v_n-v_n\|_{L^p}\le \varepsilon$. We get
\[
\|S_m v-v\|_{L^p}\le \|S_m(v-v_n)\|_{L^p}+\|S_m v_n-v_n\|_{L^p}+\|v_n-v\|_{L^p}
\le (C_p+2)\varepsilon,
\]
which proves the $L^p$ convergence.

For the $H^s_D$ convergence, apply the previous argument to $(I+A)^{s/2}v\in L^2$,
using the commutation $(I+A)^{s/2}S_m=S_m(I+A)^{s/2}$ from Lemma
\ref{lem:Sm-L2-app}.
\end{proof}

\begin{remark}[Why this appendix is essential in the BT model, and where KV differs]
In the BT damping model, the only purpose of $S_m$ is to replace the sharp spectral
projector by a smooth multiplier which remains uniformly bounded on $L^p$ and is
compatible with nonhomogeneous Strichartz estimates. This is precisely the mechanism
used in Section 3 of the paper. In contrast, for Kelvin--Voigt damping
$-\mathrm{div}(a(x)\nabla u_t)$, one additionally needs frequency-local estimates
(Bernstein-type inequalities) and commutator bounds with multiplication by $a(x)$.
\end{remark}

\begin{remark}
Smooth spectral multipliers play a crucial role in energy--critical
dispersive problems. In contrast with sharp Galerkin projections,
operators of the form $S_m=\chi(\sqrt{-\Delta}/m)$ remain uniformly
bounded in $L^p$ spaces, preventing the pathological growth typically
associated with abrupt spectral truncations. This uniform boundedness
is precisely what enables the derivation of Strichartz estimates
independent of the approximation parameter.
\end{remark}


\end{document}